\documentclass[b5paper,reqno]{amsart}
\usepackage{amsfonts,lingmacros,tree-dvips, amsmath, amssymb, graphicx, mathrsfs}

\usepackage[all]{xy}

\newdir{ >}{!/6pt/@{ }*:(1,0.2)@_{>}*:(1,-0.2)@^{>}}

\theoremstyle{plain}

\newtheorem{theorem}{Theorem}[section]
\newtheorem{lemma}[theorem]{Lemma}

\theoremstyle{definition}
\newtheorem{definition}[theorem]{Definition}

\newtheorem{counter example}[theorem]{Counter Example}

\newtheorem{corollary}[theorem]{Corollary}

\numberwithin{equation}{section}

\DeclareMathAlphabet{\mathscr}{OT1}{pzc}{m}{it} 
\begin{document}
\Large{
		\title{CATEGORY BASES THAT ARE EQUIVALENT TO TOPOLOGIES}
		
		\author[A.C.Pramanik]{Abhit Chandra Pramanik}
	\address{\large{Department of Pure Mathematics, University of Calcutta, 35, Ballygunge Circular Road, Kolkata 700019, West Bengal, India}}
	\email{\large{abhit.pramanik@gmail.com}}
		
		\author[S. Basu]{Sanjib Basu}
		\address{\large{Department of Mathematics,Bethune College,181 Bidhan Sarani}}
		\email{\large{sanjibbasu08@gmail.com}}
		
		\author[A. Deb Ray]{Atasi Deb Ray}
		\address{\large{Department of Pure Mathematics, University of Calcutta, 35, Ballygunge Circular Road, Kolkata 700019, West Bengal, India}}
		\email{\large{debrayatasi@gmail.com}}

		\thanks{The first author thanks the CSIR, New Delhi – 110001, India, for financial support}
	\begin{abstract}
	In view of the fact that many of the most familiar examples of category bases are equivalent to some topology, it is natural to ask whether category bases are always topological in nature. The answer is `no'. In this paper, we show that under certain circumstances, a category base can be equivalent to a topology. So this work may be considered a continuation of similar type of works done earlier in this area.
	\end{abstract}
\thanks{}
	\maketitle

\section{INTRODUCTION}
There is a topology in the real line such that sets having Baire property in this topology coincide with Lebesgue measurable sets and first category sets coincide with sets having Lebesgue measure zero. A similar situation occurs with Marczewski sets [11] as well.\\

Since the family of Lebesgue measurable sets and that of Marczewski sets both form catrgory bases, it follows from above that these category bases are topological in nature. A basic question may be raised : which category bases are equivalent to topological spaces ?. Morgan [3] showed that in a category base, if every region contains a minimal region, then it is equivalent to some topology. It was also shown by Schilling [11], Piotrowski and Szymanski in an abstract  presented to AMS (for reference see [11], [1]) that category bases whose cardinalities are less than or equal to $\aleph_{1}$ are equivalent to  topologies. In the same abstract, Piotrowski and Szymanski also announced the construction of a category base which is not equivalent to any topology. They did not however presented a proof of this phenomenon. Using the concept of lower density, Detlefson and Szymanski [1] provided characterization of category bases equivalent to topologies. Thus in the past, it was shown that under certain circumstances, a category base can be equivalent to a topology and this paper is a continuation of such type of works done earlier in this area. Here with the help of a condition that generalizes condition used by Morgan, it is shown that a category base can be equivalent to a topology.
\section{PRELIMINARIES AND RESULTS}
In a series of papers [4], [5], [6], [7] etc, Morgan developed the theory of category bases which generalizes the notion of a topology and within which many of the analogies between Lebesgue measure and Baire category can be unified. We start by recalling some of the basic Definitions and results relevant to our purpose. These may be found collectively in the book [8].

\begin{definition}
A category base is a pair $(X,\mathcal{C}$) where $X$ is a non-empty set and $\mathcal{C}$ is a family of subsets of $X$ whose non-empty members called regions, satisfy the following set of axioms:
\begin{enumerate}
\item Every point of $X$ belongs to some region; i.e., $X = \cup$ $\mathcal{C}$.
\item Let $A$ be a region and $\mathcal{D}$ be any non-empty family of disjont regions having cardinality less than the cardinality of $\mathcal{C}$.\\
i) If $A$ $\cap$ ( $\cup$ $\mathcal{D}$) contains a region, then there is a region $D \in$ $\mathcal{D}$ such that $A$ $\cap$ $D$ contains a region.\\
ii) If $A$ $\cap$ ( $\cup$ $\mathcal{D}$) contains no region, then there is a region $B$ $\subseteq$ $A$ that is disjoint from every region in $\mathcal{D}$.
\end{enumerate}
\end{definition}

\begin{definition}
In a category base ($X,\mathcal{C}$), a set $S$ is called $\mathcal{C}$-singular (or, singular) if every region contains a subregion which is disjoint from $S$. A set which can be expressed as countable union of singular sets is called $\mathcal{C}$-meager (or, meager). Otherwise, it is called abundant.
\end{definition}
 The class of all $\mathcal{C}$-meager sets is here denoted by the symbol $\mathcal{M}$($\mathcal{C}$). In a category base which is equivalent to some topology $\tau$, this class coincides with the class of all first category sets expressed here by the symbol $\mathcal{M}$($\tau$).

\begin{definition}
	In a category base ($X,\mathcal{C}$), a set $S$ is called a $\mathcal{C}$-Baire (or, Baire) set if in every region, there is a subregion in which either $S$ or its complement $X$ $-$ $S$ is meager.
\end{definition}
 The class of all $\mathcal{C}$-Baire sets is here denoted by $\mathcal{B}$($\mathcal{C}$). In a category base which is equivalent to some topology $\tau$, this class is same as the class of all sets having Baire property [10] which is here denoted by $\mathcal{B}a$($\tau$). \\
 
Thus a category base (X, $\mathcal{C}$) is  equivalent to a topological space ($X, \tau$) iff $\mathcal{M}$($\mathcal{C}$) = $\mathcal{M}$($\mathcal{\tau}$) and $\mathcal{B}$($\mathcal{C}$) = $\mathcal{B}a$($\tau$).\\

In a category base ($X, \mathcal{C}$),
\begin{theorem}
	The intersection of two regions either contains a region or is a singular set. 
\end{theorem}
\begin{theorem}(The Fundamental Theorem)
	Every abundant set is abundant everywhere in some region. This means that for any abundant set $S$, there exists a region $C$ in every subregion D of which $S$ is abundant.\\
\end{theorem}
The above theorem is a generalization of the famous Banach category theorem in topological spaces [8].\\

The Fundamental Theorem stated above and the Definition of a Baire set leads to the following 
\begin{corollary}
	If $B$ is an abundant Baire set, then there exists a region $C$ such that $C$ $-$ $B$ is meager. 
\end{corollary}
Now we define 
\begin{definition}
	An operator $D$ : $\mathcal{P}(X)\mapsto$ $\mathcal{P}(X)$ on the set of all subsets of $X$ of ($X, \mathcal{C}$) satisfying the conditions : 
	\begin{enumerate}
		\item $D(X) = X$ and $D(S) = \phi$ if $S$ is singular.
		\item $D$($S$ $\cup$ $T$) = $D(S)$ $\cup$ $D(T$).
	\end{enumerate}
\end{definition}
It is easy to check that if $S$ $\subseteq$ $T$ then $D(S$) $\subseteq$ $D(T)$ for in this case we may write $T$ = $S$ $\cup$ ($T$ $-$ $S$) and use (ii). Moreover, $D(S$) $-$ $D(T$) $\subseteq$ $D$($S$ $-$ $T$) for any two sets $S$ and $T$ and $D$( $\bigcap\limits_{i} S_i$) $\subseteq$ $\bigcap\limits_{i}$ $D(S_i$) for any arbitrary collection \{$S_i$\} of sets in $X$.\\

A set $S$ in a category base ($X, \mathcal{C}$) is called locally meager at a point $x$ [8] if for every region $A$ containing $x$, there is a subregion $B$ of $A$ containing $x$ such that $S$ $\cap$ $B$ is meager. Otherwise, $S$ is called locally abundant at $x$ [8]. In otherwords, $S$ is locally abundant at $x$ if there is a region $A$ containing $x$ such that for every subregion $B$ of $A$ and $x \in$ $B$ implies that $S$ $\cap$ $B$ is abundant.\\

If $S$ is locally abundant at $x$, then $x$ is called a cluster point of $S$ [8]. Setting $C(S)$ as the set of all cluster points of $S$, it is easy to check that in particular the operator $S$ $\mapsto$ $C(S)$ satisfies conditions (i) and (ii) of Definition 2.7.\\

If $D$ is in general an operator given by Definition 2.7, then $\tau$($D$) = \{ $S$ : $D(X$ $-$ $S$) $\subseteq$ $X$ $-$ $S$\} forms a topology which we call the $D$-topology associated with $(X, \mathcal{C}$) and whose members will be called the $D$-open sets. In particular, when $D(S)$ is the set $C(S)$ of all cluster points of $S$, this topology is referred to as the basic topology on $X$ [8].\\

Our theorem states 
\begin{theorem} \label{thrm:neat}
	In a category base $(X, \mathcal{C}$), if we assume that every region contains a non-empty $D$-open set, then $(X, \mathcal{C}$) is equivalent to the $D$-topology. In otherwords, $\mathcal{M}$($\mathcal{C}$) = $\mathcal{M}$($\tau(D))$ and $\mathcal{B}$($\mathcal{C}$) = $\mathcal{B}a(\tau(D))$.
\end{theorem}
 \begin{lemma}
 	Every non-empty $D$-open set is an abundant Baire set in $(X, \mathcal{C}$).
 	\begin{proof}
 		Let $S$ be a non-empty $D$-open set. If $S$ is a meager then by condition (i) of Definition 2.7, $D(S)$ = $\phi$.\\
 		Therefore,
 		\begin{equation*} \label{eq1}
 		\begin{split}
     X &= D(X)\\
 			& = D(X) - D(S)\\
 			& \subseteq D(X - S)\\
 			& \subseteq X - S \\
 		&	- ~a ~contradiction.
 		\end{split}
 		\end{equation*}
 		
 		Suppose $S$ is not Baire. Then there is a region $C$ in which both $S$ and $X$ $-$ $S$ are abundant everywhere.\\
 		Consequently, 
 		\begin{equation*} \label{eq1}
 			\begin{split}
 				S \cap C &\subseteq D(X - S)\\ 
 				&\subseteq X - S\\
 				& -~a~contradiction.
 			\end{split}
 		\end{equation*}
 		
 		Therefore, $S$ is an abundant Baire set.
 		\end{proof}
 \end{lemma}

\begin{proof} [Proof of  theorem \ref{thrm:neat}]
We first show that $\mathcal{M}$($\mathcal{C}$) = $\mathcal{M}$($\tau(D)$).\\
Let $M$ $\in$ $\mathcal{M}$($\mathcal{C}$) and $U$ be any non-empty $D$-open set. Then $U$ $-$ $M$ is also a non-empty $D$-open set because
\begin{equation*} \label{eq1}
\begin{split}
	 D(X-(U - M)) &= D(X - U) \cup D(U \cap M) \\
	& = D(X - U) ~as ~D(U \cap M) \subseteq D(M) = \phi\\
	& \subseteq X - U\\
	& \subseteq X - (U - M)
\end{split}
\end{equation*}
But $M$ $\cap$ ($U$ $-$ $M$) = $\phi$. So the $D$-open set $U$ has a non-empty $D$-open subset disjoint with $M$ proving that $M$ $\in$ $\mathcal{M}$($\tau(D)$). Hence $\mathcal{M}$($\mathcal{C}$) $\subseteq$ $\mathcal{M}$($\tau(D)$).\\

To prove the reverse inclusion, we need only to show that every $D$-nowhere dense set belongs to $\mathcal{M}$($\mathcal{C}$).\\

Let $A$ be a $D$-nowhere dense set but $A$ $\notin$ $\mathcal{M}$($\mathcal{C}$). By the fundamental theorem (Theorem 2.5), there exists a region $C$ such that $A$ is abundant everywhere in $C$.\\

By our assumption, $C$ contains a non-empty $D$-open set $U$ and this $D$-open set contains a non-empty $D$-open set $V$ such that $A$ $\cap$ $V$ = $\phi$. By Lemma 2.9, $V$ is an abundant Baire set in $(X, \mathcal{C}$), so by Corollary 2.6, there exists a region $C^\prime$ such that $C^\prime - V \in \mathcal{M}(\mathcal{C})$. Then the intersection $C$ $\cap$   $C^\prime$ being abundant, by Theorem 2.4, contains a region $E$ such that $E$ $\cap$ $A$ $\in$ $\mathcal{M}$($\mathcal{C}$) which is impossible.\\

Hence  $\mathcal{M}$($\tau(D)$) $\subseteq$  $\mathcal{M}$($\mathcal{C}$).\\

Next we show that $\mathcal{B}$($\mathcal{C}$) = $\mathcal{B}a$($\tau(D)$). Let $A$ $\in$ $\mathcal{B}$($\mathcal{C}$) and $T$ be any non-empty $D$-open set. By Lemma 2.9, $T$ is an abundant Baire set and so by Corollary 2.6, there exists a region $C$ such that $C$ $-$ $T$ is meager. Again, by Definition of Baire set in $(X, \mathcal{C}$), $C$ contains a region $C^\prime$ such that either $A$ $\cap$ $C^\prime \in \mathcal{M}(\mathcal{C})$ or $C^\prime - A \in \mathcal{M}(\mathcal{C})$. By our assumption, $C^\prime \supseteq$ $T^\prime \in \tau(D)$ ($T^\prime$ $\neq \phi$) where we may asssume that $T^\prime \subseteq$ $T$ simply by replacing $T$ $\cap$ $T^\prime$ by $T^\prime$. Since $\mathcal{M}$($\tau(D)$) =  $\mathcal{M}$($\mathcal{C}$), so either $A$ $\cap$ $T^\prime$ $\in$ $\mathcal{M}$($\tau(D)$) or $T^\prime$$-$ $A$ $\in$ $\mathcal{M}$($\tau(D)$). Let \{$T_\alpha$\}$_{\alpha < \Omega }$ be an well-ordering of all $D$-open sets in $X$. Thus every $D$-open set $T_\alpha$ contains a $D$-open set $T_{\alpha } ^\prime$ such that either $A$ $\cap$ $T_{\alpha } ^\prime$ or $T_{\alpha } ^\prime$ $-$ $A$ is of first category in the $D$-topology. Since $\bigcup\limits_{\alpha < \Omega} T_\alpha$ $-$ $\bigcup\limits_{\alpha < \Omega}T_{\alpha } ^\prime$ is $D$-nowhere dense, so from the Union Theorem (pg 82, [2]), it follows that $A$ $\in$ $\mathcal{B}a$($\tau(D)$). Hence 
$\mathcal{B}$($\mathcal{C}$) $\subseteq$ $\mathcal{B}a$($\tau(D)$).\\

Let $A$ $\in$ $\mathcal{B}a$($\tau(D)$) and $C$ be any region. If $A$ $\in$ $\mathcal{M}$($\tau(D)$), then $A$ $\cap$ $C$ $\in$ $\mathcal{M}$($\tau(D)$) = $\mathcal{M}$($\mathcal{C}$). If $A$ $\notin$ $\mathcal{M}$($\tau(D)$), then we may write $A$ = ($H$ $-$ $Q$) $\cup$ $R$ [10] where $H$($\neq$ $\phi$) $\in$ $\tau(D)$ and $Q$, $R$ $\in$ $\mathcal{M}$($\tau(D)$). By our assumption, $C$ $\supseteq$ $T$ $\in$ $\tau(D)$ ($T$ $\neq$ $\phi$) and by Lemma 2.9, $T$ is an abundant Baire set in $(X, \mathcal{C}$). Consequently, by Corollary 2.6, there exists a region $E$ such that $E$ $-$ $T$ $\in$ $\mathcal{M}$($\mathcal{C}$). Moreover, as the intersection $C$ $\cap$ $E$ is abundant, so by Theorem 2.4, there exists a region $F$ such that $F$ $\subseteq$ $C$ $\cap$ $E$. Two cases may occur. Either $H$ $\cap$ $T$ = $\phi$ or $H$ $\cap$ $T$ $\neq$ $\phi$. If $H$ $\cap$ $T$ = $\phi$, then $A$ $\cap$ $F$ $\in$ $\mathcal{M}$($\mathcal{C}$). If $H$ $\cap$ $T$ $\neq$ $\phi$, choose $\phi$ $\neq$ $O$ ( $\in$ $\tau(D)$) $\subseteq$ $H$ $\cap$ $T$, $E^\prime$ and $F^\prime$ as the regions such that $E^\prime$ $-$ $O$ $\in$ $\mathcal{M}$($\mathcal{C}$) and $F^\prime$ $\subseteq$ $C$ $\cap$ $E^\prime$. But then $F^\prime$ $-$ $A$ $\in$ $\mathcal{M}$($\mathcal{C}$). Thus every region has a subregion in which either $A$ or its complement is meager proving that $A$ $\in$ $\mathcal{B}$($\mathcal{C}$). Hence $\mathcal{B}a$($\tau(D)$) $\subseteq$ $\mathcal{B}$($\mathcal{C}$). This proves the theorem.
	\end{proof}	
The equivalence theorem of Morgan [3] states that if $(X, \mathcal{C}$) is a category base in which every region contains a minimal region, then $(X, \mathcal{C}$) is equivalent to the basic topology associated with $(X, \mathcal{C}$). Since any union of minimal regions is an open set in the basic topology [3], so Morgan's hypothesis stated in terms of minimal region gives rise to a situation in which our assumption is fulfilled.\\

Another example where the situation similarly leads to the fullfilment of our assumption may be illustrated by the topology introduced by Martin [9] in certain measure spaces. \\

Let $(X, \mathcal{L}$, $m$) be a totally finite, complete measure space in which $m(X)$ = 1 and there is a collection $\mathcal{K}$ of sequences \{$K_n$\} of sets from $\mathcal{L}$ satisfying the property that for each $x \in$ $X$ there exists at least one sequence \{$K_n$\} $\in$ $\mathcal{K}$ such that 
\begin{enumerate}
	\item $x \in$ $K_n$ for each $n$.
	\item $m$($K_n$) $\mapsto$ $0$ as $n$ $\mapsto$ $\infty$.
\end{enumerate}

In [9], the definition of upper (resp, lower) density denoted by $D^{-*}(E,x)$ (resp, ${D_{-}}^{*}(E,x)$) of any set $E$ at a point $x$ was given in terms of sequences from $\mathcal{K}$.\\

Martin showed that the family of subsets of $X$ defined by $\mathcal{U}$ = \{$U$ :  $D^{-*}$($X$ $-$ $U$, $x$) = 0 for all $x$ $\in$ $U$\} forms a topology on $X$ which is equivalent to the topology induced by the operator $D$ where $D(E)$ = \{$x$ :   $D^{-*}(E, x$)$ > 0$\}. The measure $m$ being a category measure for $(X, \mathcal{U}$) (Theorem 5.1, [9]), the category base is equivalent to the $D$-topology. Moreover, in this circumstances, Martin also proved (Lemma 4.8, [9]) that every $m$-measurable set of positive measure (which is a region in the category base $(X, \mathcal{L}$)) contains a $D$-open set.

\bibliographystyle{plain}

	\end{document}